\documentclass[11pt]{amsart}

\usepackage[utf8]{inputenc}
\usepackage[english]{babel}

\usepackage{geometry}
\usepackage{anysize}
\usepackage{amsfonts}
\usepackage{amssymb}
\usepackage{amsmath}
\usepackage{amsthm}
\usepackage{tikz-cd}
\usepackage{hyperref}
\usepackage{amscd}
\usepackage[all]{xy}
\usepackage{enumerate}
\usepackage{multirow}
\newfont{\sheaf}{eusm10 scaled\magstep1}

\newtheorem{definition}{Definition}[section]

\newtheorem{proposition}{Proposition}[section]
\newtheorem{prop}{Proposition}[section]
\newtheorem{corollary}[proposition]{Corollary}
\newtheorem{lemma}[proposition]{Lemma}
\newtheorem{theorem}[proposition]{Theorem}
\newtheorem*{thm*}{Theorem}
\newtheorem*{cor*}{Corollary}

\newtheorem{remark}{Remark}[proposition]

\DeclareMathOperator{\Ker}{ker}
\DeclareMathOperator{\Supp}{Supp}
\DeclareMathOperator{\Sing}{Sing}

\DeclareMathOperator{\Ext}{Ext}
\DeclareMathOperator{\ext}{ext}
\DeclareMathOperator{\Hom}{Hom}
\DeclareMathOperator{\Pic}{Pic}

\DeclareMathOperator{\Rk}{rk}
\DeclareMathOperator{\rk}{rk}
\DeclareMathOperator{\Coker}{coker}

\newcommand{\C}{\mathbb{C}}

\newcommand{\PP}{\mathbb{P}}
\newcommand{\OO}{\mathcal{O}}
\newcommand{\mc}[1]{\mathcal{#1}}

\newcommand{\mf}[1]{\mathfrak{#1}}

\newcommand{\im}{im}
\newcommand{\U}{\mathcal{U}_C}
\newcommand{\SU}{\mathcal{SU}_C}
\newcommand{\mI}{\mathcal{I}_2}
\newcommand{\mF}{\mathcal{F}}
\newcommand{\mHom}{\mathcal{H}om}
\newcommand{\mV}{\mathcal{V}}

\newcommand{\id}{id}
\newcommand{\mM}{\mathcal{M}}
\newcommand{\ev}{ev}
\newcommand{\evv}{ev_V}

\begin{document}
\title{Genus 2 curves and generalized theta divisors}
\thanks{Both authors are partially supported by INdAM - GNSAGA. We would like to thank Alessandro Verra for useful comments on the preliminary version of this paper.}

\author{Sonia Brivio, Filippo F. Favale}
\address[Sonia Brivio]{Department of Mathematics, University of Milano-Bicocca, Via Roberto Cozzi, 55, 20125 Milano (MI)}
\email{sonia.brivio@unimib.it}
\address[Filippo F. Favale]{Department of Mathematics, University of Milano-Bicocca, Via Roberto Cozzi, 55, 20125 Milano (MI)}
\email{filippo.favale@unimib.it}


\begin{abstract}
In this paper we investigate generalized theta divisors $\Theta_r$ in the moduli spaces $\U(r,r)$ of semistable vector bundles on a curve $C$ of genus $2$. We provide a desingularization $\Phi$ of $\Theta_r$ in terms of a projective bundle $\pi:\PP(\mV)\to\U(r-1,r)$ which parametrizes extensions of stable vector bundles on the base by $\OO_C$. Then, we study the composition of $\Phi$ with the well known theta map $\theta$. We prove that, when it is restricted to the general fiber of $\pi$, we obtain a linear embedding.
\end{abstract}
\maketitle

\pagestyle{plain}
\baselineskip=12pt
\section*{Introduction}

\noindent Theta divisors play a fundamental role in the study of moduli spaces of semistable vector bundles on curves. First of all, the classical notion of theta divisor of the Jacobian variety of  a curve can be generalized to higher rank.
Let $C$ be a smooth, irreducible, complex, projective curve of genus $g \geq 2$. The study of isomorphism classes of stable vector bundles  of fixed rank $r$ and degree $n$ goes back to Mumford. The compactification of this moduli space is denoted by ${\mathcal U}_C(r,n)$ and has been introduced by Seshadri. In the particular case when the degree is equal to $r(g-1)$ it admits a natural Brill-Noether locus 
${\Theta}_r$,  which is called the  {\it  theta divisor} of ${\mathcal U}_C(r,r(g-1))$. Riemann's singularity Theorem extends to $\Theta_r$, see \cite{Lz}.

\noindent When we restrict our attention to semistable vector bundles of rank $r$ and fixed determinant $L \in \Pic^{r(g-1)}(C)$, we have the moduli space $\SU(r,L)$ and a Brill-Noether locus ${\Theta}_{r,L}$ which is called the  {\it theta divisor} of ${\mathcal SU}_C(r,L)$.
The line bundle associated to ${\Theta}_{r,L}$ is the ample generator ${\mathcal L}$ of the Picard variety of $\SU(r,L)$,  which is called the {\it determinant line bundle}, see \cite{DN}.
\vspace{2mm}

\noindent For semistable vector bundles with integer slope, one can also introduce the notion of {\it associated theta divisor}. In particular for a stable $E \in \SU(r,L)$ with $L\in \Pic^{r(g-1)}(C)$ we have that the set
$$\{N\in \Pic^{0}(C) \, |\, h^0(E\otimes N)\geq 1 \}$$
is either all $\Pic^0(C)$ or an effective divisor $\Theta_E$ which is called the theta divisor of $E$. Moreover the map which associates to each bundle $E$ its theta divisor $\Theta_E$ defines a rational map
$$ \theta \colon \SU(r,L) \dashrightarrow \vert r \Theta_M \vert,$$
where $\Theta_M$ is a translate of the canonical theta divisor of $\Pic^{g-1}(C)$ and $M$ is a line bundle such that $M^{ \otimes r} = L$. Note that the indeterminacy locus of $\theta$ is given by set the vector bundles which does not admit a theta divisor.
\vspace{2mm}

\noindent Actually, this map is defined by the determinant line bundle $\mathcal L$, see \cite{BNR89} and it has been studied by many authors.
It has been completely described for $r= 2$ with the contributions of many authors. On the other hand, when $r\geq 3$, very little is known. 
In particular, the genus $2$ case seems to be interesting. First of all, in this case we have that $\dim \SU(r,L) = \dim \vert r \Theta_M \vert$. For $r=2$ it is proved in \cite{N-R} that
$\theta$ is an isomorphism, whereas, for $r=3$ it is a double covering ramified along a sextic hypersurface (see \cite{Or}). For $r\geq 4$ this is no longer a morphism, and it is generically finite and dominant, see \cite{B2} and \cite{BV}. 
\vspace{2mm}

\noindent In this paper, we will consider a smooth curve $C$ of genus $2$. In this case,  the theory of extensions of vector bundles allows us to give a birational description of the Theta divisor ${\Theta_r}$  as a projective
bundle over the moduli space ${\mathcal U}_C(r-1,r)$.
Our first result is Theorem \ref{THM:RIS1} which can be stated as follows
\begin{thm*}
There exists a vector bundle ${\mathcal V}$ on ${\mathcal U}_C(r-1,r)$ of rank $2r-1$ whose fiber at the point $[F] \in {\mathcal U}_C(r-1,r)$ is $\Ext^1(F,\OO_C)$. Let ${\mathbb P}(\mathcal V)$ be the associated projective bundle and $\pi \colon {\mathbb P}(\mathcal V) \to {\mathcal U}_C(r-1,r)$ the natural projection. Then the map
$$ \Phi \colon {\mathbb P}(\mathcal V) \to \Theta_r$$
sending $[v]$ to the vector bundle which is  extension of $\pi([v])$ by $\OO_C$, is a birational morphism. 
\end{thm*}

\noindent In particular, notice that this theorem gives a desingularization of $\Theta_r$ as ${\mathbb P}(\mathcal V)$ is smooth. As a corollary of the above Theorem  we have, (see \ref{cor1}),  that ${\Theta}_{r,L}$ is birational to a projective bundle over the moduli space $\SU(r-1,L)$ for any $r\geq 3$.  
This has an interesting consequence (see Corollary \ref{cor2}):
\begin{cor*}
${\Theta}_{r,L}$ is a rational subvariety of $\SU(r,L)$. 
\end{cor*}
\noindent The proof of the Theorem and its corollaries can be found in Section \ref{SEC:2}.
\vspace{2mm}

\noindent 
The second result of this paper is contained in Section \ref{SEC:3} and it involves the study of the restriction of $\Phi$ to the general fiber $\PP_F=\pi^{-1}([F])$ of $\pi$ and its composition with the theta map. 
The main result of this section is Theorem \ref{THM:MAIN2} which can be stated as follows:
\begin{thm*}
For a general stable bundle $F \in \SU(r-1,L)$  the map
$$ \theta \circ \Phi|_{\PP_F} \colon \PP_F \to \vert r \Theta_M \vert $$
is a linear embedding.
\end{thm*}
\noindent In the proof we are actually more precise about the generality of $F$: we describe explicitely a open subset of the moduli space $\SU(r-1,L)$ where the above theorem holds. Let us stress that one of the key argument in the proof involves the very recent result about the stability of secant bundles $\mathcal{F}_2(E)$ over the two-symmetric product of a curve, see \cite{BD}.
\vspace{2mm}

\noindent It would be interesting to extend the above results to a curve of genus $g\geq 3$, but the generalization is not straightforward as one can think. First of all, in order to have  a  projective bundle    over the moduli space ${\mathcal U}_C(r-1,r(g-1))$, as in theorem \ref{THM:RIS1},  we need to assume  that $r-1$ and $r(g-1)$ are coprime. Nevertheless, also in these hypothesis $\Phi$ is no more a morphism (see Remark \ref{REM:NEEDG2} for more details). Finally, in order to generalize the second result, one need to consider secant bundles over $g$-symmetric product of a curve. Unfortunately, in this case, it is not known whether the secant bundle $\mathcal{F}_g(E)$ is stable when $E$ is so, and this is one of the key argument of our proof in the case $g=2$.

%-------------------------------------------------------
%-------------------------------------------------------

% SECTION 1 : Background and known results

%-------------------------------------------------------
%-------------------------------------------------------

\section{Background and known results}
\label{SEC:1}
\noindent In this section we recall some definitions and useful results about generalized Theta divisors, secant bundles and $2$-symmetric product of curves that we will use in the following sections.

%-------------------------------------------------------
%-------------------------------------------------------

\subsection{Theta divisors}
\label{SUBSEC:1}
$\,$\vspace{1mm}

\noindent Let $C$ be a smooth, irreducible, complex, projective curve of genus $ g = 2$.
For any  $r \geq 2$ and for any $n \in {\mathbb Z}$, let ${\mathcal U}_C(r,n)$ denote  the moduli space of semistable vector bundles on the curve $C$ 
with rank $r$ and degree $n$. It is a normal, irreducible,  projective variety  of dimension $r^2+1$, 
whose points are $S$-equivalence classes of  semistable vector bundles of rank $r$ and degree $n$; we recall that   two vector bundles   are called  to be  $S$-equivalent if  they have isomorphic  graduates, where   the graduate  $gr(E)$ of $E$ is the polystable bundle defined by a Jordan-Holder filtration of $E$, see \cite{Ses82} and \cite{LP}.  We denote by $\mathcal{U}_C(r,n)^s$  the open subset correponding to isomorphism classes of stable bundles. For $r=2$ one has that $\mathcal{U}_C(r,n)$ is smooth, whereas, for $r\geq 3$ one has
$$\Sing (\mathcal{U}_C(r,n))  = \mathcal{U}_C(r,n) \setminus \mathcal{U}_C(r,n)^s.$$  
Moreover, $\mathcal{U}_C(r,n) \simeq {\mathcal U}_C(r,n')$ whenever $n'-n= kr$, with $k \in {\mathbb Z}$, and $\mathcal{U}_C(r,n)$ is  a fine moduli space if and only if $r$ and $n$ are coprime. 
\vspace{2mm}

\noindent For any line bundle $L \in \Pic^{n}(C)$, let ${\mathcal SU}_C(r,L)$ denote the moduli space of semistable vector bundles on $C$ with rank $r$ and fixed determinant $L$. 
These moduli spaces  are  the fibres of the natural map ${\mathcal U}_C(r,n) \to \Pic^{n}(C)$ which associates to each vector bundle its determinant. 
\vspace{2mm}

\noindent When $n = r$, we consider the following Brill-Noether loci:
$${\Theta}_r  = \{  [E] \in {\mathcal U}_C(r,r) \  \vert \  h^0(gr(E)) \geq 1  \ \},$$
$${\Theta}_{r,L}  = \{  [E] \in {\mathcal SU}_C(r,L) \  \vert \  h^0(gr(E)) \geq 1  \ \},$$
where $[E]$ denotes $S-$equivalence class of $E$. 
Actually, ${\Theta}_r$ (resp. ${\Theta}_{r,L}$)  is an integral Cartier divisor which is called the {\it theta divisor } of ${\mathcal U}_C(r,r)$ (resp. ${\mathcal SU}(r,L)$), see \cite{DN}. The  line bundle $\mc{L}$ associated to ${\Theta}_{r,L}$ is called the {\em determinant bundle} of ${\mathcal SU}_C(r,L)$ and it is the generator of its Picard variety.
We denote by ${\Theta_r}^s \subset \Theta_r$ the open  
subset of stable points. Let  $[E] \in {\Theta_r}^s$, then  the multiplicity of $\Theta_r$  at the point $[E]$ is $h^0(E)$, see \cite{Lz}. This implies:
$$\Sing(\Theta_r^s)=\{[E]\in \Theta_r^s \,|\, h^0(E)\geq 2\}.$$

\noindent For semistable vector bundles with integer slope we can introduce the notion of theta divisors as follows.
Let $E$ be  a semistable vector bundle  on $C$ with integer slope $m = \frac{\deg E}{r}$.
%Recall that, if $E$ is a vector bundle and $\mc{L}$ is a line bundle, then
%$$\deg(E\otimes \mc{L})=\deg(\mc{L})\Rk(E)+\deg(E)$$
%and
%$$\mu(E\otimes \mc{L})=\mu(E)+\deg(\mc{L}).
\vspace{2mm}

\noindent The tensor product defines a morphism
$$\mu \colon {\mathcal U}_C(r,rm) \times \Pic^{1-m}(C) \to {\mathcal U}_C(r,r)$$
sending $([E], N) \to [E \otimes N]$.    
\vspace{2mm}

\noindent The intersection 
${\mu}^*{\Theta}_r \cdot ([E] \times \Pic^{1-m}(C))$ is either 
an effective divisor 
$\Theta_E$ on $Pic^{1-m}(C)$ which is called the {\it theta divisor } of $E$, or all $([E] \times \Pic^{1-m}(C))$, and in this case we will say that $E$ {\it does not admit theta divisor}. For more details see \cite{B}.
\vspace{2mm}

\noindent Set theoretically we have
$$ \Theta_E = \{ N \in \Pic^{1-m}(C)  \  \vert \  h^0(gr(E )\otimes N)  \geq 1  \ \}.$$
For all $L\in \Pic^{rm}(C)$ fixed we can choose a line bundle $M \in \Pic^{m}(C)$ such that $L= M^{\otimes r}$.
If $[E] \in \SU(r,L)$, then $\Theta_E \in \vert r \Theta_M \vert$ where 
$$\Theta_M  = \{ N \in \Pic^{1-m}(C)  \  \vert h^0(M \otimes N ) \geq 1 \   \}$$
is a translate of the canonical theta divisor $\Theta \subset \Pic^{g-1}(C)$. 
This defines a rational map, which is called the  {\em theta map} of $\SU(r,L)$
\begin{equation}
\label{EQ:THETAMAP}
\xymatrix{
\SU(r,L)\ar@{-->}[r]^-{\theta} & |r\Theta_M|.
}
\end{equation}
As previously recalled $\theta$ is the map induced by the determinant bundle $\mc{L}$ and the points $[E]$ which do not admit theta divisor give the indeterminacy locus of $\theta$. Moreover $\theta$ is an isomsorphism for $r = 2$, it is a double covering ramified along a sextic hypersurface for $r=3$. For $r \geq 4$  it is no longer a morphism: it is generically finite and dominant. 
\hfill\par

%-------------------------------------------------------
%-------------------------------------------------------

\subsection{$2$-symmetric product of curves}
\label{SUBSEC:2}
$\,$\vspace{1mm}

\noindent Let $C^{(2)}$
denote the $2$-symmetric product of $C$,
parametrizing effective divisors $d$ of degree $2$ on the 
curve $C$.
It is well known that $C^{(2)}$  is a smooth projective surface, see  \cite{A-C-G-H}. 
It is the quotient of the product $C \times C$ by the action of the symmetric group ${\mathcal S}_2$; we denote by 
$$\pi \colon C \times C \to C^{(2)}, \quad \pi(x,y) = x+y,$$ 
the quotient map, which is  a double covering  of $C^{(2)}$, ramified along the diagonal $\Delta \subset C \times C$.
\vspace{2mm}

\noindent Let $N^1(C^{(2)})_{\mathbb Z}$ be the Neron-Severi group of $C^{(2)}$, i.e. the quotient group of numerical equivalence classes of divisors on 
$C^{(2)}$.  
For any $p \in C$, let 's consider the embedding
$$i_p \colon C \to C^{(2)}$$ 
sending $q \to q+p$, we denote the image by $C + p$ and 
we denote by $x$  its numerical class in  $N^1(C^{(2)})_{\mathbb Z}$.
Let $d_2$ be the diagonal map  
$$d_2 \colon C \to C^{(2)}$$
sending $q \to 2q$. Then  $d_2 (C) = \pi(\Delta) \simeq C$,    we denote by  $\delta $ its   
numerical class in  $N^1(C^{(2)})_{\mathbb Z}$.
Finally, let's consider the Abel map 
$$A \colon C^{(2)} \to \Pic^2(C)\simeq J(C)$$
sending $p+q \to O_C(p+q)$.   Since $g(C) = 2$, it is well known that actually
$C^{(2)}$ is the blow up  of $\Pic^2(C)$ at $\omega_C$ with exeptional divisor 
$$\mf{E} = \{ d \in C^{(2)} \vert  \  \OO_C(d) \simeq {\omega}_C \} \simeq {\mathbb P}^1.$$
This implies that: 
$${K}_{C^{(2)}} =  A^* (K_{\Pic^2(C)}) + \mf{E} = \mf{E},$$ 
since $K_{\Pic^2(C)}$  is trivial.
% being $\Pic^2(C)$ an abelian variety.
 \hfill\par\noindent
Let $\Theta \subset J(C)$ be the theta divisor, its pull back $A^*(\Theta)$ is an effective divisor on 
$C^{(2)}$, we denote by   $\theta $ its  numerical class in $N^1(C^{(2)})_{\mathbb Z}$.   
It is well known that $\delta = 2( 3x - \theta)$, or, equivalently,
\begin{equation}
\theta=3x-\frac{\delta}{2}.
\end{equation}
If $C$ is a general curve of genus $2$ then $N^1(C^{(2)})_{\mathbb Z}$ is generated by the classes ${x}$ and $\frac{\delta}{2}$ (see \cite{A-C-G-H}). The Neron-Severi lattice is identified by the relations
$$x\cdot x = 1, \quad x\cdot  \frac{\delta}{2} = 1, \quad \frac{\delta}{2} \cdot \frac{\delta}{2} = -1.$$
%It is worth to mention also that
%\begin{equation}
%K_{\CS}=\mf{E}\simeq 2x-\frac{\delta}{2} 
%\end{equation}
%\hfill\par
%\noindent It can be shown that 
%\begin{equation}
%\frac{\delta}{2}\quad \mbox{ and }\quad 2x-\frac{\delta}{2}=\mf{E}
%\end{equation}
%are minimal generators for the extremal rays of the effective cone. Their dual rays are generated, %respectively by 
%\begin{equation}
%v_1=x+\frac{\delta}{2}\quad \mbox{ and }\quad v_2= 3x-\frac{\delta}{2}=\theta
%\end{equation}
%and are the extremal rays of the nef cone $NE(\CS)$ of $\CS$.
%\begin{center}
%\includegraphics[width=0.95\textwidth]{NS}
%\end{center}

%\noindent We will recall that $e(\CS)=\deg(c_2(\CS))=1$ (see \ref{EQ:CHERNCLC2}). As we have just seen %that $\deg(c_1(\CS)^2)=-1$, we have that $\chi_{\CS}(\OO_{\CS})=0$ by Noether formula. As $K_{\CS}=%%\mf{E}$ is rigid, we have $q(\CS)=h^{1,0}(\CS)=2$. Then, from $e(\CS)=1$, we have also $h^{1,1}(\CS)=5$.
%Summarizing, we have
%\begin{equation}
%h^{1,0}(\CS)=2,\quad h^{2,0}(\CS)=1,\quad h^{1,1}(\CS)=5, \end{equation}
%\begin{equation*}
%e(\CS)=1,\quad \chi_{\CS}(\CS)=0.
%\end{equation*}

%-------------------------------------------------------
%-------------------------------------------------------

\subsection{Secant bundles on $2$-symmetric product of curves}
\label{SUBSEC:3}
$\,$\vspace{1mm}

\noindent Let's consider the {\it universal effective divisor} of degree $2$ of $C$:
$$ {\mathcal I}_2 = \{ (d,y) \in C^{(2)} \times C \ \vert \ y \in \Supp(d) \},$$
it is a smooth irreducible divisor on $ C^{(2)} \times C$. Let $\iota$ be the embedding of $\mI$ in $C^{(2)}\times C$, $r_1$ and $r_2$ be the natural projections of $ C^{(2)} \times C$ onto factors and $q_i=r_i\circ \iota$ the restriction to $\mI$ of $r_i$. Then $q_1$ is a surjective map of degree $2$. Denote also with $p_1$ and $p_2$  the natural projections of $C \times C$ onto factors.
\vspace{2mm}

\noindent We have a natural  isomorphism 
$$ \nu \colon C \times C \to {\mathcal I}_2, \quad (x,y) \to (x+y, y)$$
and, under this isomorphism,  the map $q_1 \colon \mI \to C^{(2)}$ can be identified with the map $\pi \colon C \times C \to C^{(2)}$. It is also easy to see that the map $q_2$, under the isomorphism $\nu$, can be identified with the projection $p_2$. We have then a commutative diagram

$$\xymatrix{
& C^{(2)}\times C\ar[ld]_-{r_1}\ar[rd]^-{r_2} &\\
C^{(2)} & \mI \ar[u]_-{\iota}\ar[l]^-{q_1} \ar[r]^-{q_2} & C\\
& C\times C \ar[ul]^-{\pi} \ar[ur]_-{p_2} \ar[u]^-{\nu}
}$$

\hfill\par\noindent
Now we will introduce the secant bundle $\mF_2(E)$ associated to a vector bundle $E$ on $C$ as well as some properties which will be useful in the sequel. For an introduction on these topics one can refer to  \cite{Sch64} or the Ph.D. thesis of E. Mistretta, whereas some interesting recent results can be found in \cite{BN12} and \cite{BD}. 
\vspace{2mm}

\noindent Let $E$ be a vector bundle of rank $r$ on $C$, we can associate to $E$ a sheaf on $C^{(2)}$ which is defined as 
\begin{equation}
\mc{F}_2(E) = q_{1*}({q_2^*(E)}).
\end{equation}
 $\mc{F}_2(E)$ is  a vector bundles of rank $2r$  which is called the {\it secant bundle associated to $E$} on $C^{(2)}$.
%The fiber of ${\mathcal F}_2(E)$ at $d \in C^{(2)}$ is actually the vector space $H^0(C, E \otimes O_d)$ and 
%$$H^0(C^{(2)}, {\mathcal F}_2(E)) \simeq H^0(C, E).$$ 
\vspace{2mm}

\noindent Let's consider the pull back of the secant bundle on $C \times C$: $ {\pi}^* {\mathcal F}_2(E)$. Outside the diagonal $\Delta \subset C \times C$ we have: 
$$ {\pi}^* {\mathcal F}_2(E)  \simeq p_1^*E \oplus p_2^*(E).$$  Actually, these bundles are related by the  following  exact sequence:
\begin{equation} 
0 \to {\pi}^* {\mathcal F}_2(E) \to p_1^*E \oplus p_2^*(E) \to  p_1^*(E)_{\vert \Delta} = p_2^*(E)_{\vert \Delta} \simeq E \to 0,\end{equation}
where the last map sends $(u,v) \to u_{\vert \Delta}-v_{\vert \Delta}$.

%\hfill\par\noindent
%By composing the inclusion $ {\pi}^* {\mathcal F}_2(E) \hookrightarrow  p_1^*E \oplus p_2^*(E) $ with %the projections to ${p_1}^*E$, we obtain  the following exact sequences
%\begin{equation}
%0 \to p_2^*(E) \otimes (- \Delta)  \to  {\pi}^* {\mathcal F}_2(E) \to p_1^*(E) \to 0,\end{equation}
%\begin{equation}
%0 \to p_1^*(E) \otimes (- \Delta)  \to  {\pi}^* {\mathcal F}_2(E) \to p_2^*(E) \to 0,\end{equation}

\hfill\par\noindent
Finally, from the exact sequence on $C^{(2)} \times C$:
$$0 \to \OO_{C^{(2)} \times C}(-\mI) \to \OO_{C^{(2)} \times C} \to \iota_*\OO_{\mI} \to 0,$$
tensoring with $r_2^*(E)$ we get:
$$0 \to r_2^*(E) (- \mI) \to r_2^*(E)  \to \iota_*(q_2^*E) \to 0,$$
where, to simplify notations, we set $r_2^*(E)  \otimes \OO_{C^{(2)} \times C}(- \mI ) = r_2^*(E) (- \mI )$ and we have used the projection formula
$$r_2^*(E)\otimes \iota_*\OO_{\mI}=\iota_*(\iota^*( r_2^*E)\otimes \OO_{\mI})=\iota_*(q_2^*E).$$

\noindent By applying  $r_{1*}$ we get
\begin{multline}
\label{EXSEQ:SECBUN1}
0 \to r_{1*}(r_2^*(E) (- \mI)) \to H^0(E) \otimes  O_{C^{(2)}}  \to \mF_2(E) \to\\
\to R^1 r_{1*}(r_2^*(E)(-\mI)) \to H^1(E) \otimes O_{C^{(2)}} \to \cdots 
\end{multline}
since  we have:
  $r_{1*}(\iota_*(q_2^*E))=q_{1*}q_2^*E=\mF_2(E)$ and 
$$R^p r_{1*}r_2^*E = H^p(E)\otimes \OO_{\C^{(2)}}.$$
Moreover, by projection formula $H^0(C^{(2)},{\mathcal F}_2(E)) \simeq H^0(C, E)$ and the map $$H^0(E) \otimes O_{C^{(2)}}  \to {\mathcal F}_2(E)$$ appearing in (\ref{EXSEQ:SECBUN1}) is actually the evaluation map of global sections of the secant bundle; we will denoted it by $ev$.
Notice that, if we have $h^1(E)=0$, the exact sequence (\ref{EXSEQ:SECBUN1}) becomes
\begin{equation}
\label{EXSEQ:SECBUN2}
\xymatrix@C=10pt{
0\ar[r] & r_{1*}(r_2^*(E) (- \mI))\ar[r] &  H^0(E)\otimes \OO_{C^{(2)}}\ar[r]^-{ev} & \mF_2(E)\ar[r] & R^1 r_{1*}(r_2^*(E)(-\mI))\ar[r] & 0
}
\end{equation}
We will call the exact sequence (\ref{EXSEQ:SECBUN1}) (and its particular case (\ref{EXSEQ:SECBUN2})) the {\it exact sequence induced by the evaluation map of the secant bundle}. If $deg E = n$, then the Chern character of ${\mathcal F}_2(E)$ is given by the following formula:
$$ ch ({\mathcal F}_2(E)) = n(1- e^{-x}) - r + r(3+\theta) e^{-x},$$
where $x$ and $\theta$ are the numerical classes defined above. 
From this we can deduce the  Chern classes of  ${\mathcal F}_2(E)$:
\begin{equation}
\label{chern1}
 c_1({\mathcal F}_2(E)) = (n-3r)x + r\theta,
\end{equation}
\begin{equation}
\label{chern2}
c_2({\mathcal F}_2(E))= \frac{1}{2} (n-3r)(n+r+1) +r^2 + 2r.
\end{equation}
%Let ${\Omega}_{C^{(2)}}^1$ be the cotangent sheaf on $C^{(2)}$, then it can be proved that actually we have:
%$${\mathcal F}_2({\omega}_C) = \Omega_{C^{(2)}}^1,$$
%in particular we have:
%\begin{equation}
%\label{EQ:CHERNCLC2}
%c_1({\mathcal F}_2({\omega}_C) ) = c_1({\Omega}_{C^{(2)}}^1) = - x + \theta, \quad c_2({\mathcal %F}_2({\omega}_C) ) =  c_2({\Omega}_{C^{(2)}}^1) = 1. 
%\end{equation}

\noindent We recall the following definition:
\begin{definition}
Let $X$ be  a smooth, irreducible, complex  projective surface and let $H$ be an ample divisor on $X$.
For a torsion free sheaf $E$ on $X$  we define the slope of $E$ with respect to $H$:
$$ \mu_H(E) = \frac{c_1(E) \cdot H}{rk(E)}.$$
$E$ is said semistable with respect to $H$ if for any non zero proper subsheaf $F$ of $E$ we have $\mu_H(F) \leq \mu_H(E)$, 
it is said stable with respect to $H$ if for any proper subsheaf $F$ with $0 < rk(F) < rk(E)$ we have $\mu_H(F) < \mu_H(E)$.
\end{definition}

\noindent One of the key arguments of the proof of our main theorems will use the following interesting result which can be found in \cite{BD}:  
\hfill\par
\begin{proposition} 
\label{PROP:STABSECB}
Let $E$ be  a semistable vector  bundle  on $C$ with rank $r$ and $deg (E) \geq  r$,  then ${\mathcal F}_2(E)$ is semistable with respect to 
the ample class $x $; if $deg(E) > r$ and $E$ is stable, then ${\mathcal F}_2(E)$ is stable too with respect to the ample class $x$.
\end{proposition}

%\begin{proposition}
%Let $E$ and $F$  be semistable vector  bundles  on $C$ with rank $r$. If ${\mathcal F}_2(E) \simeq {\mathcal F}_2(F)$, then $E \simeq F$.
%\end{proposition}

%-------------------------------------------------------
%-------------------------------------------------------

% SECTION 2 : Description of theta divisors

%-------------------------------------------------------
%-------------------------------------------------------

\section{Description of $\Theta_r$ and $\Theta_{r,L}$.}
\label{SEC:2}
\noindent In this section we will give  a description of $\Theta_r$ (resp. $\Theta_{r,L}$) which gives a natural desingularization. Fix $r\geq 3$.

\begin{lemma}
\label{LEM:FROMETOF}
Let $E$ be a stable vector bundle with $[E] \in \Theta_r$,  then there exists a vector bundle $F$  such that $E$ fit into  the following exact sequence:
\begin{equation*}
 0 \to \OO_C \to E \to F \to 0,
 \end{equation*}
with  $[F] \in {\mathcal U}_C(r-1,r)$. 
\end{lemma}

\begin{proof}
Since  $E$ is stable, $E \simeq gr(E)$ and, as $[E] \in \Theta_r$,  $h^0(E) \geq 1$. 
Let $s \in H^0(E)$ be a non zero global section, since $E$ is stable  of slope $1$, $s$ cannot be zero in any point of $C$, so it defines  an injective map of sheaves 
$$ i_s \colon \OO_C \to E$$ which induces the following  exact sequence of vector bundles:
$$ 0 \to \OO_C \to E \to F \to 0,$$
where the  quotient $F$ is a vector bundle of rank $r-1$ and degree $r$. 
We will prove  that $F$ is semistable,  hence $[F] \in \U(r-1,r)$,  which implies that it is also stable. 
\vspace{2mm}

\noindent Let $G$  be a non trivial destabilizing quotient of $F$ of degree $k$ and rank $s$ with $1 \leq  s \leq r-2$.  
Since $G$ is also a quotient of $E$, by stability of $E$ we have
$$  1 = \mu(E)  < \mu(G) \leq \mu(F) = \frac{r}{r-1},$$
i.e.
    $$ 1 <  \frac{k}{s}  \leq 1  + \frac{1}{r-1}.$$
Hence we have
$$ s < k  \leq  s + \frac{s}{r-1}$$
which is impossible since $s < r-1$. 
\end{proof}

\noindent A short exact sequence of vector bundles
\begin{equation*}
 0 \to G \to E \to F \to 0,
 \end{equation*}
is said to be {\it an extension of $F$ by $G$}, see \cite{Ati}. Recall that equivalence classes of extensions of $F$ by $G$ are parametrized by
$$H^1(\mHom(F,G)) \simeq \Ext^1(F,G);$$
where the extension corresponding to $0\in\Ext^1(F,G)$ is $G\oplus F$ and it  is called the trivial extension. Given $v\in \Ext^1(F,G)$ we will denote by $E_v$ the vector bundle which is the extension of $F$ by $G$ in the exact sequence corresponding to $v$.
Moreover, if $v_2 = \lambda v_1$ for some $\lambda\in \C^*$, we have $E_{v_1}\simeq E_{v_2}$. Lastly, recall that  $\Ext^1$ is a functorial construction so are well defined on isomorphism classes of vector bundles.

\begin{lemma}
\label{LEM:COHF}
Let  $[F]\in \U(r-1,r)$, then  $\dim \Ext^1(F,\OO_C)=2r-1$.
\end{lemma}

\begin{proof}
We have:  $\Ext^1(F,\OO_C) \simeq H^1(F^{\vee}) \simeq H^0(F\otimes \omega_C)^{\vee}$,
so by Riemann-Roch theorem:
$$\chi_C(F\otimes \omega_C)=\deg(F\otimes\omega_C)+\Rk(F\otimes\omega_C)(1-g(C))=2r-1.$$
Finally, since $\mu(F\otimes \omega_C)=3+\frac{1}{r-1}\geq 2g-1=3$, then $h^1(F \otimes {\omega}_C)= 0$. 
\end{proof}

\noindent Let $F$ be a stable bundle, with $[F] \in \U(r-1,r)$.  
The trivial extension $E_0 = \OO_C \oplus F$ gives an unstable vector bundle. However, this is the only unstable extension of $F$ by $\OO_C$  as it is  proved  in the following Lemma.

\begin{lemma} 
\label{LEM:FROMFTOE}
Let  $[F] \in \U(r-1,r)$ and $v \in \Ext^1(F,O_C)$  be a non zero vector. Then $E_v$ is a semistable vector bundle of rank $r$ and degree $r$, moreover $[E_v] \in \Theta_r$. 
\end{lemma}
\begin{proof}
By lemma \ref{LEM:COHF} $\dim \Ext^1(F,\OO_C)=2r-1>0$, let  $v\in \Ext^1(F,\OO_C)$ be a non zero vector and denote by $E_v$ the corresponding vector bundle. By construction we have an exact sequence of vector bundles
$$0\rightarrow \OO_C\rightarrow E_v\rightarrow F\rightarrow 0$$
from which we deduce  that $E_v$ has rank $r$ and degree $r$.
\vspace{2mm}

\noindent Assume that $E_v$ is not semistable. Then there exists a proper subbundle $G$ of $E_v$ with $\mu(G)>\mu(E_v)=1$. Denote with $s$ and $k$ respectively the rank and the degree of $G$. Hence we have
$$1\leq s\leq r-1\qquad k>s.$$
Let $\alpha$ be the composition of the inclusion $G\hookrightarrow E_v$ with  the surjection $\varphi:E_v\rightarrow F$, let  $K = \ker \alpha$. Then we have a commutative diagram
$$
\xymatrix{
        & 0 \ar[d] & 0 \ar[d] & 0\ar[d]\\
0\ar[r] & K \ar@{^{(}->}[r]\ar[d] & G \ar[rd]^{\alpha}\ar[r]\ar[d] & Im(\alpha)\ar[r]\ar[d] & 0\\
0\ar[r] & \OO_C\ar@{^{(}->}[r] & E_v \ar[r]_-{\varphi} & F\ar[r] & 0
}
$$
If $K=0$ then $G$ is a subsheaf of $F$,  which is stable,  so
$$\mu(G)=\frac{k}{s}<\mu(F)=1+\frac{1}{r-1}$$
and
$$s<k<s+\frac{s}{r-1},$$
which  is impossible as $1\leq s\leq r-1$. Hence we have that $\alpha$ has non trivial kernel $K$, which is a subsheaf of $\OO_C$, so   $K=\OO_C(-A)$ for some divisor $A\geq 0$ with degree $a\geq 0$. Then $Im(\alpha)$  is a subsheaf of $F$, which is stable so:
$$\frac{k+a}{s-1}<1+\frac{1}{r-1},$$
hence we have
$$s+a<k+a<s-1+\frac{s-1}{r-1}$$ and $$a<-1+\frac{s-1}{r-1}$$
which is impossible as $a\geq 0$. This proves that  $E_v$ is semistable. Finally, note  that we have 
 $h^0(E_v)\geq h^0(\OO_C)=1$, so $[E] \in \Theta_r$.
\end{proof}

% \begin{remark}
% \label{ISOM}
% The following  properties of extensions of vector bundles, which  are immediate consequence of definitions,   will be very useful  the sequel. 
% \begin{enumerate}
% \item{} Let $F$ and $F'$ be isomorphic vector bundles  and let $\sigma \colon F \to F'$ be an isomorphism.
% It  induces  an isomorphism $$\bar{\sigma} \colon  \Ext^1(F,\OO_C) \to  \Ext^1(F',\OO_C).$$
% If $E$ is  extension of $F$ by $\OO_C$ defined by $v$,  then $E$ is also extension of $F'$ by $\OO_C$ defined by $\bar{\sigma}(v)$. 
% \item{}  Let $E$ and $E'$ be isomorphic vector bundles.  
% If $E$ is  extension of $F$ by $\OO_C$,  then $E'$ is also extension of $F$ by  $\OO_C$. 
% \end{enumerate}
% \end{remark}

\noindent We would like to study extensions  of $F \in \U(r-1,r)$  by $\OO_C$ which give vector bundles of  ${\Theta}_r \setminus {\Theta}_r^s$.  Note that if $E_v $ is not stable, then there exists a proper subbundle $S$ of $E_v$ with slope $1$. 
We will prove that any such $S$ actually comes from a subsheaf  of $F$ of slope $1$. 
\vspace{2mm}

\noindent Let $[F] \in \U(r-1,r)$, observe that any proper subbsheaf  $S$ of $F$ has slope $\mu(S) \leq 1$. Indeed, let $s = \rk(S) \leq r-1$, by stability of $F$ we   have 
$$\frac{\deg(S)}{s} < 1 + \frac{1}{r-1},$$
which implies $\deg(S) < s + \frac{s}{r-1}$, hence $\deg(S) \leq s$. 
Assume that $S$ is a subsheaf of slope $1$. Then we are in one of the following cases:
\begin{itemize}
\item A subsheaf $S$  of $F$ with slope $1$ and rank $s \leq r-2$  is  a subbundle of $F$ and it is  called  {\it a maximal subbundle}  of $F$ of rank $s$. Note that any maximal subbundle $S$ is semistable and thus $[S] \in \U(s,s)$. Moreover, the set $\mM_s(F)$ of maximal subbundles of $F$ of rank $s$ has a natural scheme structure given by identifying it with a Quot-scheme (see \cite{L-N}, \cite{L-Ne} for details).
\item A subsheaf  $S$  of $F$ of slope $1$ and rank $r-1$ is obtained by {\it an elementary transformation } of $F$ at a point $p \in C$, i.e. it fits into an exact sequence as follows:
$$0 \to S \to F \to {\mathbb C}_p \to 0.$$
More precisely, let's denote with $F_p$ the fiber of $F$ at $p$, all the elementary transformations of $F$ at $p$ are parametrized by ${\mathbb P}(\Hom(F_p,\C))$. In fact, for any non zero form  $\gamma \in \Hom(F_p,\C)$,   by composing it with the restriction map $F \to F_p$, we obtain a surjective morphism $F \to {\mathbb C}_p$ and then an exact sequence
$$0 \to G_{\gamma} \to F \to {\mathbb C}_p \to 0,$$
where $G_{\gamma}$ is actually a vector bundle which is obtained by the elementary tranformation of $F$  at $p$  defined by $\gamma$. Finally, $G_{\gamma_1} \simeq G_{\gamma_2}$ if and only if $[\gamma_1]= [{\gamma_2}]$ in $ {\mathbb P}(\Hom(F_p,\C))$,  see \cite{Mar82} and \cite{B17}.
\end{itemize}

\noindent We have the following result: 
 
\begin{prop}
\label{PROP:EXTEMAXSUBF}
Let  $[F] \in \U(r-1,r)$,  $v \in \Ext^1(F,\OO_C)$ a non zero vector and $E_v$ the extension of $F$ defined by $v$. 
If $G$ is a  proper subbundle of $E_v$  of slope $1$, then $G$ is semistable and satisfies one of the following conditions: 
\begin{itemize}
\item $G$ is a maximal subbundle of $F$ and $1\leq \rk(G)\leq r-2$;
\item $G$ has rank $r-1$ and it is obtained by an elementary transformation of $F$. 
\end{itemize}
\end{prop}

\begin{proof}
Let $s=\Rk(G)=\deg(G)$.  As in the proof of Lemma \ref{LEM:FROMETOF} we can construct a commutative diagram
$$
\xymatrix{
        & 0 \ar[d] & 0 \ar[d] & 0\ar[d]\\
0\ar[r] & K \ar@{^{(}->}[r]\ar[d] & G \ar[r]\ar[d]\ar[dr]^{\alpha} & Im(\alpha)\ar[r]\ar[d] & 0\\
0\ar[r] & \OO_C\ar@{^{(}->}[r] & E_v \ar[r]_{\varphi} & F\ar[r] & 0
}
$$
form which we obtain that either $K=0$ of $K=\OO_C(-A)$ with $A\geq 0$. In the second case, let $a$ be the degree of $A$. As in the proof of Lemma \ref{LEM:FROMETOF}, we have that the slope of $Im(\alpha)$ satisfies
$$\mu(Im(\alpha))=\frac{s+a}{s-1}<1+\frac{1}{r-1}$$
which gives a contradiction 
$$0\leq a<-1+\frac{s-1}{r-1}.$$
So can assume that $K=0$, so $\alpha \colon G \to F$ is an injective map of sheaves,  we denote by  $Q$  the quotient. 
\vspace{2mm}

\noindent If $s=r-1$ we have that $Q$ is a torsion sheaf  of degree $1$, i.e. a skyscraper sheaf over  a  point with the only non trivial fiber of dimension $1$. Hence $G$ is obtained by  an elementary transformation of $F$ at a point $p \in C$. 
\vspace{2mm}

\noindent If $s\leq r-2$, we claim that $\alpha$ is an injective map of vector bundles. On the contrary, if  $G$ is not a subbundle, then    $Q$ is not locally free, so 
 there exists a subbundle  $G_f  \subset F$  containing $\alpha(G)$,  with $\Rk(G_f) = \Rk(G)$ and $deg(G_f) = deg G + b$, $b \geq 0$:  
$$
\xymatrix{
        & 0 \ar[rd] & 0 \ar[d] & 0 \ar[d] \\
&  & G \ar[d]\ar@{^{(}->}[r]\ar[dr]^{\alpha} & G_f \ar[d]& \\
0\ar[r] & \OO_C\ar@{^{(}->}[r] & E_v \ar[r]_{\phi} & F\ar[r]\ar[rd]\ar[d] & 0\\
&&&Q_f \ar[d] & Q \ar[rd]\ar@{->>}[l]\\
&&&0&&0\\
}
$$
Then, as $F$ is stable, we have:
$$\mu(G_f)=\frac{s+b}{s}=1+\frac{b}{s}<1+\frac{1}{r-1}=\mu(F),$$
hence 
$$0\leq b<\frac{s}{r-1}$$
which implies $b= 0$.

\vspace{2mm}

\noindent Finally, note that $G$ is semistable. In fact, since $\mu(G) = \mu(E_v)$, a subsheaf of $G$  destabilizing $G$ would be a subsheaf destabilizing $E_v$. 
\end{proof}

\noindent Let $S$ be a subsheaf of $F$ with slope $1$, we ask when $S$ is a subbundle of the extension $E_v$ of $F$ by $\OO_C$ defined by $v$.  

\begin{definition}
Let  $\varphi:E\rightarrow F$ and $f: S\rightarrow F$ be morphisms of sheaves. We say that $f$ can be lifted to $\tilde{f}:S\rightarrow E$ if we have a commutative diagram
$$ 
\xymatrix{
 & S \ar[d]^-{f}\ar[ld]_-{\tilde{f}}  \\
E\ar[r]_-{\varphi} & F 
}
$$
we say that $\tilde{f}$ is a  lift of $f$.
\end{definition}

\begin{lemma} 
\label{lift}
Let $F  \in \U(r-1,r)$, $v \in \Ext^1(F,\OO_C)$  be a non zero vector and $E_v$  the extension of $F$ defined by $v$.
Let  $S$ be a vector bundle of slope $1$ and $\iota \colon S \to F$ be an injective  map of sheaves.
Then $i$ can be lifted to $ E_v$ if and only if $v \in \Ker H^1(\iota^*)$ where 
$$H^1(i^*) \colon H^1(C, {\mathcal Hom}(F,\OO_C)) \to H^1(C,{\mathcal  Hom}(S,\OO_C))$$ 
is the map  induced by $i$. 
If $v \in \Ker H^1(\iota^*)$ we will say that $v$ extends $\iota$. 
\end{lemma}
\noindent For the proof see \cite{N-R}. The above lemma allows us to prove the following result: 

\begin{prop}
Let $[F]\in \U(r-1,r)$. Then:
\begin{itemize}
\item Let  $G_{\gamma}$ be the elementary transformation of $F$ at $p\in C$ defined by $[\gamma]\in \PP(F_p^{\vee})$, there exists a unique $[v] \in \PP(\Ext^1(F,\OO_C))$ such that the inclusion  $G_{\gamma} \hookrightarrow F$ can be lifted to $E_v$.
\item Let $S$ be a maximal subbundle of $F$ of rank $s$ and  $\iota:S\hookrightarrow F$ the inclusion, then the set of classes $[v]$ which extend $\iota$ is  a linear subspace of $\PP(\Ext^1(F,\OO_C))$ of dimension $2r-2s-2$. 
\end{itemize}
In particular, for any maximal subbundle of $F$ and for any elementary transformation, we obtain at least an extension of $F$ which is in $\Theta_r \setminus \Theta_r^s$. 
\end{prop}

\begin{proof}
Let's start with the case of elementary transformation. We are looking for the extensions of $F$ by $\OO_C$  such that  there exists a lift $\tilde{\iota} \colon G_{\gamma} \to E_v$ such that the diagram
$$
\xymatrix{
 & & & G_{\gamma} \ar@/_2pc/@{^{(}->}[dl]
\ar@{^{(}->}[d]^{\iota}& \\
 0\ar[r] & \OO_C\ar@{^{(}->}[r]\ar[d] & E_v \ar[r]_{\phi_v}\ar[d] & F\ar[r]\ar[d]^{\gamma} & 0\\
0\ar[r] & \OO_C\ar[r]\ar[d] & \OO_C(p)\ar[r]\ar[d] & \C_p\ar[d]\ar[r] & 0 \\
 & 0 & 0 & 0 \\
}
$$
commutes. By Lemma \ref{lift}, there exists $\tilde{\iota}$  if and only if the class of the extension $E_v$ lives in the kernel of $H^1(\iota^*)$ in the diagram
\begin{equation}
\xymatrix{
\Hom(F,\OO_C)\ar@{^{(}->}[r]\ar[d]^{\iota^*} & \Hom(F,E_v) \ar[r]^{\phi_v^*}\ar[d]^{\iota^*} & \Hom(F,F) \ar[r]^{\delta_v}\ar[d]^{\iota^*} & \Ext^1(F,\OO_C) \ar[d]^{H^1(\iota^*)}\\
\Hom(G_{\gamma},\OO_C)\ar@{^{(}->}[r] & \Hom(G_{\gamma},E_v) \ar[r]_{\phi_v^*} & \Hom(G_{\gamma},F) \ar[r]_{\delta} & \Ext^1(G_{\gamma},\OO_C)\\
}
\end{equation}
If we apply the functor $\mHom(-,\OO_C)$ to the vertical exact sequence we obtain the exact sequence
$$0\rightarrow F^{\vee}\rightarrow G_{\gamma}^{\vee}\rightarrow \C_p \rightarrow 0$$
from which we obtain 
$$\cdots \to H^1(F^{\vee})\rightarrow H^1(G_{\gamma}^{\vee})\rightarrow 0.$$
In particular, the map $H^1(\iota^*)$ is surjective so its kernel has dimension 
\begin{multline}
\dim(\ker(H^1(\iota^*)))=\ext^1(F,\OO_C)-\ext^1(G_{\gamma},\OO_C)=\\
=h^0(F\otimes \omega_C)-h^0(G_{\gamma}\otimes \omega_C)=2r-1-2(r-1)=1.
\end{multline}
Hence there exist only one possible extension which extend $\iota$.
\vspace{2mm}

\noindent Let $S$ be a  maximal subbundle  of $F$ of rank $s$, $1\leq s\leq r-2$, and let $\iota:S\rightarrow F$ the inclusion. By Lemma \ref{lift},  we have that the set of $[v]$ which extends  $\iota$ lifts is $\PP(\ker(H^1(\iota^*))$.
As in the previous case,  one can verify that  $H^1(\iota^*)$ is surjective and 
\begin{multline}
\dim (\ker(H^1(\iota^*)))=\ext^1(F,\OO_C)-\ext^1(S,\OO_C)=\\
=h^0(F\otimes \omega_C)-h^0(S \otimes \omega_C)=2r-1-2(s)=2r-2s -1.
\end{multline}
\end{proof}

\noindent The above properties of extensions allow us to give the following description of theta divisor of $\U(r,r)$: 

\begin{theorem}
\label{THM:RIS1}
There exists a vector bundle $\mV$  on $\U(r-1,r)$ of rank $2r-1$ whose fiber at the point $[F] \in \U(r-1,r)$ is 
$\Ext^1(F,\OO_C)$. Let $\PP(\mV)$ be the associated projective bundle and 
$\pi \colon \PP(\mV) \to \U(r-1,r)$ the  natural projection. Then, the map 
$$ \Phi \colon \PP(\mV) \to {\Theta}_r$$
sending $[v]$ to $[E_v]$, where $E_v$ is the extension of $\pi([v])$  by $\OO_C$ defined by $v$, is a birational morphism. 
\end{theorem}
\begin{proof}
As $r$ and $r-1$ are coprime,  there exists a Poincar\'e   bundle ${\mathcal P}$ on $\U(r-1,r)$, i.e. ${\mathcal P}$ is a vector bundle on $C \times \U(r-1,r)$ such that 
 ${\mathcal P}\vert_{ C \times [F]} \simeq F$ for any $[F] \in \U(r-1,r)$, see \cite{Ram73}.  
 Let $p_1$ and $p_2$ denote the projections of $C \times \U(r-1,r)$ onto factors. 
Consider on $C \times \U(r-1,r)$ the vector bundle $p_1^*(\OO_C)$,  note that ${p_1^*(\OO_C)}\vert_{ C \times [F]} \simeq \OO_C$, for any $[F] \in \U(r-1,r)$. 
Let consider on $\U(r-1,r)$ the first direct image of the sheaf $\mHom({\mathcal P},p_1^*(\OO_C))$, i.e. the sheaf
\begin{equation}
\label{V}
\mV = R^1{p_2}_* {\mathcal Hom}({\mathcal P},p_1^*(\OO_C)).
\end{equation}
For any $[F] \in \U(r-1,r)$ we have
$$\mV_{[F]} = H^1(C,\mHom({\mathcal P},p_1^*(\OO_C))\vert_{ C \times [F]})= H^1(C,\mHom(F,\OO_C))=\Ext^1(F,\OO_C)$$
which, by lemma \ref{LEM:COHF}, has dimension $2r-1$. Hence we can conclude that $\mV$ is a vector bundle on $\U(r-1,r)$ of rank $2r-1$ whose fibre at $[F]$ is actually $\Ext^1(F,\OO_C)$.  Let's consider the projective bundle associated to $\mV$ and the natural projection map 
$$\pi:\PP(\mV) \to \U(r-1,r).$$
Note that for any  $[F] \in \U(r-1,r)$ we have:
$$ H^0(C, \mHom({\mathcal P},p_1^*(\OO_C))\vert_{C \times [F]})= H^0(C, \mHom(F,\OO_C))= H^0(C,F^*) = 0,$$
since $F$ is stable with positive slope. Then by \cite[Proposition 3.1]{N-R}, there exists a vector bundle $\mathcal{E}$ on $C \times \mV$ such that 
for any point $v \in \mV$ the restriction 
$\mathcal{E}\vert_{ C \times v}$ is naturally identified with the extension $E_v$ of $F$ by $\OO_C$ defined by $v\in \Ext^1(F,\OO_C)$ which, by lemma \ref{LEM:FROMFTOE} is semistable and has sections, unless $v=0$. Denote by $\mV_0$ the zero section of the vector bundle $\mV$, i.e. the locus parametrizing trivial extensions by $\OO_C$. Then $\mV\setminus \mV_0$ parametrize a family of semistable extensions of elements in $\U(r-1,r)$ by $\OO_C$. This implies that the map sending $v\in \mV\setminus\mV_0$ to $[E_v]$ is a morphism. Moreover this induces a morphism 
$$ \Phi:\PP(\mV) \to \Theta_r$$
sending $[v]\in\PP(\Ext^1(F,\OO_C))$ to $[E_v]$.  
\vspace{2mm}

\noindent Note that we have: 
$$\dim {\mathbb P}(\mV) = \dim \U(r-1,r) + 2r-2= (r-1)^2 + 1 + 2r-2 = r^2 = \dim \Theta_r.$$
Moreover, by lemma \ref{LEM:FROMETOF}, $\Phi$ is dominant so we can conclude that $\Phi$ is a generically finite morphism onto $\Theta_r$. 

\noindent In order to conclude the proof it is enough to produce an open subset $U \subset \Theta_r$  such that 
the restriction
$$\Phi\vert_{{\Phi}^{-1}(U) } \colon {\Phi}^{-1}(U) \to U$$
has degree $1$. Let $U$ be the open subset of $\Theta_r$ given by the stable classes $[E]$ with $h^0(E)=1$. Now, consider $[v_1],[v_2]\in\Phi^{-1}(U)$ and assume that $\Phi([v_1])=\Phi([v_2])=[E]$. As $h^0(E)=1$ we have that $\pi([v_1])=\pi([v_2])=[F]$ and we have a commutative diagram
$$
\xymatrix{
0 \ar[r] & \OO_C \ar[d]_{\id}\ar[r]^-{s_1} & E \ar[r]\ar[d]_{\lambda\id} & F \ar[d]_{\lambda \id} \ar[r] & 0 \\
0 \ar[r] & \OO_C \ar[r]_-{\lambda s_1} & E \ar[r] & F \ar[r] & 0 
}$$
with $\lambda\in \C^*$. But this implies that the class of the extensions are multiples so we have $[v_1]=[v_2]$ and the degree is $1$.
\end{proof}

\begin{remark}
\label{REM:NEEDG2}
\rm
\noindent We want to stress the importance of the assumption on the genus of $C$ in the Theorem. Assume that $C$ is a curve  of genus $g\geq 3$. Then one can also study extensions of  a stable vector bundle $F \in {\mathcal U}_C(r-1,r(g-1))$ by $O_C$. In order to get a projective bundle $\PP(\mathcal{V})$ parametrizing all extensions, as in theorem \ref{THM:RIS1}, we need  the existence of a  Poincaré  vector bundle $\mathcal{P}$ on the moduli space $\U(r-1,r(g-1))$. This actually exists if and only if $r-1$ and $r(g-1)$ are coprime, see \cite{Ram73} (notice that this is always true if $g=2$ and $r\geq 3$). Nevertheless, also under this further assumption, we can find extensions of $F$ by $\OO_C$ which are unstable, hence the map $\Phi$ fails to be a morphism.
\end{remark}
\hfill\par

\noindent In the proof of Theorem \ref{THM:RIS1} we have seen that the fiber of $\Phi$ over a stable point $[E]$ with $h^0(E)=1$ is a single point. For stable points it is possible to say something similar:

\begin{lemma}
Let $[E] \in \Theta_r^s$, there is a bijective morphism  
$$ \nu \colon  \PP(H^0(E)) \to {\Phi}^{-1}(E).$$
\end{lemma}

\begin{proof}
Let $s \in H^0(E)$ be a non zero global section of $E$. As in the proof of  lemma \ref{LEM:FROMETOF}, $s$ induces an exact sequence of vector bundles:
\begin{equation}
\label{s}
0 \to \OO_C \to E \to F_s \to 0,
\end{equation}
where $F_s$ is stable, $[F_s] \in \U(r-1,r)$ and $E$ is the extension of $F_s$ by a non zero vector $v_s \in \Ext^1(F_s,\OO_C)$.
By tensoring \ref{s} with $F_s^*$ and taking cohomology, since $h^0(F_s^*) = h^0(F_s^* \otimes E) = 0$,  we get:
\begin{equation}
\label{ss}
0 \to H^0(F_s^* \otimes F_s) \overset{\delta}{\to} H^1(F_s^*) \overset{\lambda_s}{\to}   H^1( F_s^* \otimes E) \to H^1(F_s^* \otimes F_s) \to 0,
\end{equation}
from which we see that  
$\langle v_s\rangle$ is  the kernel of $\lambda_s$.
\vspace{2mm}

\noindent So we have a natural map:
$$H^0(E) \setminus \{ 0 \}  \to  {\mathbb P}(\mV) $$
sending a non zero global section $s \in H^0(E)$ to $[v_s]$. 
Let $s$ and $s'$ be non zero global sections such that $s' = \lambda s$,  with $\lambda \in C^*$.  As in the proof of Theorem \ref{THM:RIS1}, it turns out that $v_{s'} = \lambda v_s$ in $\Ext^1(F,\OO_C)$.   
So we have a map:
$$\nu \colon {\mathbb P}(H^0(E)) \to {\mathbb P}(\mV) $$ 
sending $[s] \to [v_s]$, whose image is  actually ${\Phi} ^{-1}(E)$.
\vspace{2mm}

\noindent We claim that this map is   a morphism. Let $\PP^n = {\mathbb P}(H^0(E))$, with $n \geq 1$,  one can prove that  there exists a vector bundle ${\mathcal Q}$ on
${\mathbb P}^n \times C$  of rank $r-1$ such that ${\mathcal Q}\vert_{ [s] \times C} \simeq F_s$. Hence we have a morphism 
$\sigma \colon {\mathbb P}^n \to \U(r-1,r)$, sending $[s] \to [F_s]$,  and a vector bundle ${\sigma}^* \mV$ on ${\mathbb P}^n$.  Finally, 
there exists a vector bundle ${\mathcal G}$  on ${\mathbb P}^n$  with ${\mathcal G}_{[s]} = H^1(F_s^* \otimes E)$ and  a map of vector bundles:
$$\lambda \colon {\sigma}^*(\mV) \to {\mathcal G},$$
 where   ${\lambda}_{[s]}$ is the map  appearing in \ref{ss}. Since  $\langle v_s\rangle = \ker {\lambda}_s$, this implies the claim.   
\vspace{2mm}

\noindent To conclude the proof,   we show  that $\nu$ is injective. 
Let $s_1$ and $s_2$ be  global sections and assume 
that $[v_{s_1}]=[v_{s_2}]$. Then $s_1$ and $s_2$ defines two exact sequences which give two extensions which are multiples of each other. Then, there exists an isomorphism $\sigma$ of $E$ such that the diagram
$$
\xymatrix{
0 \ar[r] & \OO_C \ar[d]_{\id}\ar[r]^-{s_1} & E \ar[r]\ar[d]_{\sigma} & F \ar[d]_{\lambda \id} \ar[r] & 0 \\
0 \ar[r] & \OO_C \ar[r]_-{s_2} & E \ar[r] & F \ar[r] & 0 
}$$
is commutative. But $E$ is stable, so $\sigma=\lambda\id$. Then, clearly, $\sigma_1=\lambda\sigma_2$. 
\hfill\par
\end{proof}

\noindent Let  $L \in \Pic^r(C)$ and $\SU(r-1,L)$ be the  moduli space of stable vector bundles with determinant $L$. As we have seen, $\SU(r-1,L)$ can be seen as a subvariety of 
$\U(r-1,r)$. Let $\mV$ be the vector bundle on $\U(r-1,r)$ defined in the proof of Theorem \ref{THM:RIS1}.
Let $\mV_L$ denote the restriction of $\mV$ to $\SU(r-1,L)$. We will denote with $\pi:\PP(\mV_L)\rightarrow \SU(r-1,L)$ the projection map. Then, with the same arguments of the proof of Theorem \ref{THM:RIS1} we have the following:

\begin{corollary}
\label{cor1}
Fix $L \in \Pic^r(C)$. The map 
$$ \Phi_L :\PP(\mV_L) \rightarrow \Theta_{r,L}$$
sending $[v]$ to the extension $[E_v]$ of $\pi([v])$ by $\OO_C$ defined by $v$, is a birational morphism. 
\end{corollary}

\noindent As $r$ and $r-1$ are coprime, we have that $\SU(r-1, L)$ is a rational variety, see \cite{New75,KS}. Hence, as a consequence  of our theorem we have also this interesting corollary: 
\begin{corollary}
\label{cor2}
For any $L \in \Pic^{r}(C)$, ${\Theta}_{r,L}$ is a rational subvariety of $\SU(r,L)$.
\end{corollary}

%-------------------------------------------------------
%-------------------------------------------------------

% SECTION 3 : Composition of $\Phi|_F$ with $\theta$

%-------------------------------------------------------
%-------------------------------------------------------

\section{General fibers of $\pi$ and $\theta$ map}
\label{SEC:3}
\noindent In this section, we would restrict the morphism $\Phi$ to extensions of a general vector bundle $[F] \in \U(r-1,r)$.
First of all  we will deduce some  properties of general elements of $\U(r-1,r)$. 
\vspace{2mm}

\noindent For any vector bundle $F$, let $\mc{M}_1(F^*)$ be the scheme  of maximal line subbundles of $F^*$. Note that, if $[F]\in \U(r-1,r)$, then maximal line subbundles of $F^*$ are exactly the line subbundles of degree $-2$.

 \begin{proposition}
\label{maximal}
 Let  $r \geq 3$,  a general $[F] \in \U(r-1,r)$ satisfies the following properties:
\begin{enumerate}
\item{} if $r \geq 4$, $F$ does not admit  maximal subbundles    of rank $ s \leq r-3$;
\item{}   $F$ admits finitely many  maximal subbundles   of rank $ r-2$;
\item{} we have ${\mathcal M}_{r-2}(F)   \simeq {\mathcal M}_1(F^*).$
 \end{enumerate}
 \end{proposition}
 \begin{proof}
For any $1 \leq s \leq r-2$ let's consider the following locus:
$$T_s = \{ [F] \in \U(r-1,r) \quad\vert\quad  \exists  S \hookrightarrow F \   \mbox{ with }   \deg(S) = \rk(S) = s \}.$$
The set $T_s$ is locally closed, irreducible of dimension
$$ \dim T_s = (r-1)^2 +1 + s(s-r+2),$$
see \cite{L-Ne}, \cite{R-T}.
If $r \geq 4$ and $s \leq r-3$, then $\dim T_s < \dim \U(r-1,r)$, which proves (1). 
\hfill\par\noindent
(2) 
 Let $r \geq 3$   and $s = r-2$.  Then actually $T_{r-2} = \U(r-1,r)$ and  a general $[F] \in \U(r-1,r)$ has finitely many maximal subbundles of rank $r-2$. See \cite{L-Ne}, \cite{R-T} for a proof in the general case and 
 \cite{L-N} for $r=3$, where actually the property actually holds for any $[F] \in \U(2,3)$. 
\hfill\par\noindent 
(3)   
Let  $[F ] \in \U(r-1,r)$ be a general element and $[S] \in {\mathcal M}_{r-2}(F)$, then $S$ is semistable and we have an exact sequence
$$ 0 \to S \to F \to Q \to 0$$
with $Q \in \Pic^2(C)$. Moreover $S$ and $Q$ are general in their moduli spaces as in \cite{L-Ne}. This implies that 
$\Hom(F,Q) \simeq {\mathbb C}$. In fact, 
by  taking the dual of the above sequence and  tensoring with $Q$ we obtain
$$0 \to Q^* \otimes Q \to F^* \otimes Q  \to S^* \otimes Q \to  0$$
and, passing to cohomology we get
$$ 0 \to H^0(O_C) \to  H^0( F^* \otimes Q )  \to H^0(S^* \otimes Q)  \to \cdots .$$
Since $S$ and $Q$ are general $h^0(S^* \otimes Q) = 0$ and we can conclude 
$$\Hom(F,Q) \simeq H^0(F^* \otimes Q) \simeq H^0(O_C) = \C.$$
We have a natural map 
$q \colon {\mathcal M}_{r-2}(F) \to {\mathcal M}_1(F^*)$
sending $S$ to $Q^*$. The map $q$ is surjective as any maximal line subbundle $Q^* \hookrightarrow F^*$ gives a surjective map
$\phi \colon F \to  Q$ whose kernel is a maximal subbundle $S$  of $F$. 
The map is also injective.  Indeed, assume that $[S_1]$ and $[S_2]$  are maximal subbundles such that $q(S_1)= q(S_2)= Q^*$.
Then $S_1 = \ker \phi_1$ and  $S_2 = \ker \phi_2$, with $\phi_i \in \Hom(F,Q) \simeq {\mathbb C}$.
This implies that  $\phi_2 = \rho \phi_1$, $\rho \in \C^*$, hence $S_1 \simeq S_2$. Moreover, note that the above construction works for any flat family of semistable maximal subbundles of $F$, hence  $q$ is a morphism. Finally, the same construction gives a morphism $q':\mathcal{M}_{1}(F^*)\to \mathcal{M}_{r-2}(F)$ which turns out to be the inverse of $q$. This concludes the proof of (3).
\end{proof}

\begin{lemma}
\label{LEM:EVsecant}
For any $r \geq 3$ and $[F] \in \U(r-1,r)$, let $\ev$ be the evaluation map of the secant bundle $\mF_2(F\otimes \omega_C)$. If $\mc{M}_1(F^*)$ is finite, then $\ev$ is generically surjective and its degeneracy locus $Z$ is the following:
$$ Z = \{ d\in C^{(2)} \,\vert\, \OO_C(-d) \in {\mathcal M}_1(F^*) \}.$$
Moreover, $Z \simeq {\mathcal M}_1(F^*)$ if and only if $h^0(F) =1$;  if $h^0(F) \geq 2$ then $Z = \mf{E} \cup Z'$, where ${\mf E}=|\omega_C|$ (see Section \ref{SEC:1}) and $Z'$ is a finite set. 
\end{lemma}
\begin{proof}
As we have seen in section \ref{SEC:1}, ${\mathcal F}_2(F \otimes {\omega}_C)$ is a vector bundle of rank $2r-2$ on $C^{(2)}$ and 
$H^0( C^{(2)}, {\mathcal F}_2(F \otimes {\omega}_C)) \simeq H^0(C, F \otimes {\omega}_C)$. Recall that the evaluation map of the secant bundle of $F\otimes \omega_C$ is the map
$$\ev \colon H^0(F \otimes {\omega}_C) \otimes O_{C^{(2)}} \to {\mathcal F}_2(F \otimes {\omega}_C)$$
and is such that, for any  $d \in C^{(2)}$, $ev_d$ can be identified with the restriction map
$$ H^0(F \otimes {\omega}_C)  \to (F\otimes\omega_C)_d, \quad 
s \to s\vert_{ d}.$$
Observe that 
\begin{equation}
\label{EQ:HOMH1}
H^1(F \otimes {\omega}_C(-d)) \simeq H^0(F^* \otimes O_C(d))^* \simeq \Hom(F,O_C(d))^*.
\end{equation}
Note that for any $d\in C^{(2)}$ we have: 
$$\Coker(\ev_d) \simeq H^1(F \otimes {\omega}_C(-d)),$$
hence $\ev_d$ is not surjective if and only if $\Hom(F,O_C(d)) \not= 0$, that is $O_C(-d)$ is a maximal line subbundle of $F^*$.
If $F$ has finitely many maximal line subbundles  we can conclude that $\ev$ is generically surjective  and its degeneracy locus is the following: 
$$Z = \{ d \in C^{(2)} \,\vert\, \Rk(\ev_d) < 2r-2\} = 
\{ d \in C^{(2)} \,\vert\, O_C(-d) \in {\mathcal M}_1(F^*) \}.$$
Let $a \colon C^{(2)} \to \Pic^{-2}(C)$ be the map sending $d \to O_C(-d)$, $a$ is the composition of $A \colon C^{(2)} \to \Pic^{2}(C)$ sending $d$ to $\OO_C(d)$ with the isomorphism 
$\sigma \colon \Pic^{2}(C) \to \Pic^{-2}(C)$ sending $Q \to Q^*$.  Then $Z = a^{-1}({\mathcal M}_1(F^*))$. Note that 
\begin{equation}
Z \simeq {\mathcal  M}_1(F^*)\Longleftrightarrow \omega_C^{-1} \not\in{\mathcal M}_1(F^*) \Longleftrightarrow h^0(F)= 1.
\end{equation}
If $h^0(F) \geq 2$, then $ {\mf E} = \vert {\omega}_C \vert  \subset Z$ and this concludes the proof.
\end{proof}

\begin{remark}
Under the hypothesis of Lemma \ref{LEM:EVsecant}, the evaluation map fit into an exact sequence
\begin{equation}
\label{exact3}
 0 \to M \to H^0(F \otimes {\omega}_C) \otimes O_{C^{(2)}} \to {\mathcal F}_2(F \otimes {\omega}_C) \to T \to 0,
\end{equation}
where $M$ is a line bundle and $\Supp(T) = Z$. 
\end{remark}

\begin{remark} 
\rm  Let $[F] \in\U(r-1,r)$ be a general vector bundle, by proposition \ref{maximal},  $ {\mathcal M}_1(F^*) \simeq {\mathcal M}_{r-2}(F)$ is a finite set,  
 moreover $\Hom(F,O_C(d)) \simeq {\mathbb C}$ when $ \OO_C(-d) \in {\mathcal M}_1(F^*)$.  Finally, $[F]$ being general, we have $h^0(F)=1$ and this implies 
$$Z \simeq {\mathcal M}_1(F^*).$$
Taking the dual sequence of \ref{exact3} we have:
$$ 0 \to { {\mathcal F}_2(F \otimes {\omega}_C)}^* \to  H^0(F \otimes {\omega}_C)^* \otimes O_{C^{(2)}} \to M^* \otimes J_Z \to 0,$$
and computing Chern classes we obtain:
$$ c_1(M^*) = c_1( {\mathcal F}_2(F \otimes {\omega}_C)),$$
$$ c_1({\mathcal F}_2(F \otimes {\omega}_C)^*) c_1(M^*) + c_2({\mathcal F}_2(F \otimes {\omega}_C)^*) + l(Z) = 0,$$
from which we deduce:
$$ l(Z) = c_1({\mathcal F}_2(F \otimes {\omega}_C))^2 - c_2({\mathcal F}_2(F \otimes {\omega}_C)).$$
We have:
$$c_1({\mathcal F}_2(F \otimes {\omega}_C)) = x + (r-1)\theta, \quad  
c_2({\mathcal F}_2(F \otimes {\omega}_C)) = r^2 + 2r -2,$$
so we obtain:
$$ l(Z) = (r-1)^2.$$
This gives  the cardinality of ${\mathcal M}_{r-2}(F)$ and of ${\mathcal M}_1(F^*)$. This formula actually holds also for $F \in \U(r,d)$,  see \cite{G,L} for $r= 3$ and \cite{O-T,O} for $r\geq 4$. 
\end{remark}

%Let $[F] \in \U(r-1,r)$ and let $\PP_F=\PP(\Ext^1(F,\OO_C)) $ be the fibre of the projective bundle %$\PP(\mV)$ at $[F]$.
%Let's consider the restriction of the morphism $\Phi$:
%\begin{equation}
%\label{PhiF}
%\Phi_F = \Phi|_{\PP(\Ext^1(F,\OO_C)) } \colon {\mathbb P}(\Ext^1(F,\OO_C))  \to \Theta_r.
%\end{equation}
%Let $L = det F \in Pic^{r}(C)$, then the image of $\Phi_F$ is a closed subset of $ \Theta_{r,L}$.

\noindent The stability properties of the secant bundles,   on the two-symmetric product of a curve, allow us to prove the following.

\begin{proposition}
\label{thetadivisor}
Let $r \geq 3$ and $[F] \in \U(r-1,r)$ with $h^0(F) \leq 2$.  
If $\mc{M}_1(F^*)$ is finite, then every non trivial extension of $F$ by $\OO_C$  gives a vector bundle which admits theta divisor.
\end{proposition} 
\begin{proof}
Let $E$ be an  extension of $F$ by $\OO_C$ which does not admit a theta divisor. Hence
$$ 0 \to \OO_C \to E \to F \to 0,$$
and, by tensoring with ${\omega}_C$ we obtain
\begin{equation}
\label{E}
 0 \to {\omega}_C \to {\tilde E} \xrightarrow{\psi}  {\tilde F} \to 0, 
\end{equation}
where, to simplify the notations, we have set ${\tilde E} = E \otimes {\omega}_C$ and ${\tilde F} = F \otimes {\omega}_C$. 
Note that $\tilde E$ does not admit theta divisor too, hence 
$$ \{ l \in \Pic^{-2}(C) \vert h^0({\tilde E} \otimes l) \geq 1 \} = \Pic^{-2}(C).$$
This implies that $\forall d \in C^{(2)}$ we have $h^0(\tilde{E} \otimes O_C(-d)) \geq 1$ too.
Let's consider the cohomology exact sequence induced by the exact sequence (\ref{E})
$$0 \to H^0({\omega}_C) \to H^0(\tilde{E}) \xrightarrow{\psi_0} H^0({\tilde F}) \to H^1({\omega}_C) \to 0, $$
where we have used $h^1({\tilde E}) = 0$ as $\mu({\tilde E}) = 3 \geq 2$.
Let's consider the  subspace of $H^0({\tilde F})$ given by the image of $\psi_0$, i.e. 
$$V=\psi_0(H^0(\tilde{E})).$$ 
In particular $\dim V = h^0({\tilde F}) -1 = 2r-2$ so $V$ is an hyperplane.
\vspace{2mm}

\noindent {\bf Claim}:  For any $d \in C^{(2)}\setminus \mathfrak{E}$ we have $V \cap H^0({\tilde F} \otimes O_C(-d)) \not= 0$.
\vspace{2mm}

\noindent In fact, by tensoring the exact sequence (\ref{E})  with $O_C(-d)$ we have:
$$ 0 \to {\omega}_C \otimes O_C(-d)  \to {\tilde E} \otimes O_C(-d) \to  {\tilde F} \otimes O_C(-d)  \to 0,$$
for a   general $d \in C^{(2)}$, then passing to cohomology we obtain the inclusion:
$$ 0 \to H^0({\tilde E} \otimes O_C(-d)) \to H^0({\tilde F} \otimes O_C(-d)),$$
which implies the claim since $h^0({\tilde E} \otimes O_C(-d)) \not= 0$.
\hfill\par\noindent
Let $\ev\colon  H^0(\tilde{F}) \otimes O_{C^{(2)}} \to {\mathcal F}_2({\tilde F})$ be the evaluation map of the secant bundle associated to 
${\tilde F}$ and consider its restriction to $V \otimes O_{C^{(2)}}$.
We have a diagramm as follows:
\begin{equation}
\label{DIAG:BIG}
\xymatrix{
0 \ar[r] &
	\ker(\evv) \ar@{^{(}->}[r]\ar@{^{(}->}[d]  &
    V\otimes \OO_{C^{(2)}} \ar@{->>}[r]^-{\evv}\ar@{^{(}->}[d] & 
    \im(\evv) \ar[r]\ar@{^{(}->}[d] &
    0 \\
0 \ar[r] & 
	M \ar@{^{(}->}[r]\ar[d] &
    H^0(\tilde{F})\otimes \OO_{C^{(2)}} \ar[r]_-{\ev} &
    \mF_2(\tilde{F}) \ar[r] & 
    T \ar@{->>}[r] &
    0 \\
 &
	Q
}\end{equation} 
where $M$ is a line bundle, T has support on $Z$ as  in Lemma \ref{LEM:EVsecant}.
For any $d\in C^{(2)}$ we have that the stalk of $\ker(\evv)$ at $d$ is
$$\Ker(\evv)_d=\ker\left((\evv)_d:V\otimes\OO_d\rightarrow \mF_2(\tilde{F})_d\right)=H^0(\tilde{F}\otimes \OO_C(-d))\cap V.$$
Notice that, as a consequence of the claim, 
$$\dim\left( H^0(\tilde{F}\otimes \OO_C(-d))\cap V\right)\geq 1$$ for any non canonical divisor $d$. Hence $\Ker(\evv)$ is a torsion free sheaf of rank $1$. 
For all $d\in C^{(2)}\setminus Z$ we have $h^0(\tilde{F}\otimes \OO_C(-d))=1$, hence, for these points, we have
$$\ker(\evv)_d=H^0(\tilde{F}\otimes \OO_C(-d)).$$
In particular, as $M$ and $\ker(\evv)$ coincide outside $Z$, we have that the support of $Q$ is cointained in $Z$.
\vspace{2mm}

\noindent In order to conclude the proof we will use the stability property of the secant bundle. With this aim, recall that, as seen in \ref{LEM:EVsecant},  $c_1(\mc{F}(\tilde{F}))=x+(r-1)\theta$ and thus, $c_1(\mc{F}(\tilde{F}))\cdot x=2r-1$. In particular, if $H$ is an ample divisor with numerical class $x$ we have
\begin{equation}
\mu_H(\mc{F}(\tilde{F}))=\frac{2r-1}{2r-2}.
\end{equation}
We will distinguish two cases depending on the value of $h^1(F)$.
\vspace{2mm}

\noindent {\bf Assume that  $h^0(F) = 1$}. In this case $Z \simeq {\mathcal M}_1(F^*)$ is a finite set (see Lemma \ref{LEM:EVsecant}). 
The support of $T$ is finite too so we have 
$$c_1(\im(\evv))=-c_1(\ker(\evv))=-c_1(M)=c_1(\mF_2(\tilde{F})).$$
Hence, we can conclude that $\im(\evv)$ is a proper subsheaf of the secant bundle with rank $2r-3$ and with the same first Chern class. Hence
\begin{equation}
\mu_H(\im(\evv))=\dfrac{c_1(\im(\evv))\cdot x }{2r-3}=\dfrac{x\cdot(x+(r-1)\theta)}{2r-3}=\dfrac{2r-1}{2r-3}
\end{equation}
but this contraddicts Proposition \ref{PROP:STABSECB}. This conclude this case.
\vspace{2mm}

\noindent {\bf Assume that  $h^0(F) = 2 $}. In this case $Z=\mf{E}\cup Z'$ with $Z'$ of dimension $0$ by  Lemma \ref{LEM:EVsecant}. Recall that the numerical class of $\mf{E}$ in $C^{(2)}$ is $\theta-x$ (see Section \ref{SEC:1}).
Observe that $Supp(T) = {\mf E} \cup Z'$ and for any $d \in {\mf E}$ we have: $\dim T_d= 1$. 
From the exact sequence of the evaluation map of the secant bundle we obtain:
$$c_1(M) = {\mf E} - c_1(\mF_2({\tilde F})).$$
Since $Supp(Q) \subset Z$, we distinguish two cases depending on its dimension.
\vspace{2mm}

\noindent (a) If $\dim Supp(Q) = 0$, then we have
$$c_1(\im(\evv))=-c_1(\ker(\evv))=-c_1(M),$$
hence $c_1(\im(\evv)) = c_1(\mF_2(\tilde F))- \mf{E}$. Then
\begin{equation}
\mu_H(\im(\evv))=
\dfrac{x\cdot(2x+(r-2)\theta)}{2r-3}=\dfrac{2r-2}{2r-3}
\end{equation}
But this is impossible since the secant bundle is semistable by Proposition \ref{PROP:STABSECB}. 
\vspace{2mm}

\noindent (b) If $\dim Supp(Q) =1$, since $Supp(Q) \subset Z$ and ${\mf E}$ is  irreducible, then 
$Supp(Q) = \mf{E} \cup Z'$, with $Z'$ finite or empty.  Observe that for any $d \in \mf{E}$ we have:
$\dim Q_d = 1$. 
So we have
$$ c_1(\im(\evv))=-c_1(\ker(\evv))=-c_1(M) + {\mf E},$$
hence $c_1(\im(\evv)) = c_1(\mF_2(\tilde{F}))$ and we can conclude as above.
\end{proof}

%---------------------------------------------
%---------------------------------------------
%---------------------------------------------
\noindent Fix a line bundle $L = M^{\otimes r}$, with $M \in \Pic^1(C)$. 
Let $[F] \in \SU(r-1,L)$, we consider  the fibre of the projective bundle $ \pi \colon \PP(\mV) \to \U(r-1,r)$ at $[F]$:
$$\PP_F=\PP(\Ext^1(F,\OO_C)) = {\pi}^{-1}([F]) \simeq {\mathbb P}^{2r-2},$$  
and the restriction  of the morphism $\Phi$ to $\PP_F$:
\begin{equation}
\label{PhiF}
\Phi_F = \Phi|_{\PP_F} \colon \PP_F  \to \Theta_{r,L}.
\end{equation}
By Corollary \ref{cor1} 
the map 
$$ \Phi_L \colon \PP(\mV_L) \to \Theta_{r,L}$$
is a birational morphism. Then, there exists a non empty open subset $U \subset \Theta_{r,L}$ such that
$${\Phi_L}_{\vert {\Phi_L}^{-1}(U)} \colon {\Phi_L}^{-1}(U) \to U$$
is an isomorphism. Hence, for general $F \in \SU(r-1,L)$ the intersection 
$\Phi^{-1}(U) \cap \PP_F$ is a non empty open subset of $\PP_F$ and 
$$\Phi_F \colon \PP_F \to {\Theta}_{r,L}$$ is a birational morphism onto its image.
\vspace{2mm}

\noindent Recall that 
\begin{equation}
\xymatrix{
\SU(r,L)\ar@{-->}[r]^-{\theta} & |r\Theta_M|.
}
\end{equation}
is the rational map which sends $[E]$ to $\Theta_E$. Note that if $F$ is generic then, by Proposition \ref{thetadivisor}, we have that $\theta$ is defined in each element of $\im(\Phi_F)$ so it makes sense to study the composition of $\Phi_F$ with $\theta$ which is then a morphism:
$$
\xymatrix{
\PP_{F}\ar@/_1pc/@{->}[dr]_{\theta \circ \Phi_F} \ar[r]^-{\Phi_F} & \Theta_{r,L}\ar@{-->}[d]^-{\theta} \\
 & |r\Theta_M|
}$$

\noindent We have the following result:
\begin{theorem}
\label{THM:MAIN2}
For a general stable bundle $F \in \SU(r-1,L)$  the map
$$ \theta \circ \Phi_F \colon \PP_F \to \vert r \Theta_M \vert $$
is a linear embedding.
\end{theorem}
\begin{proof}
As previously noted, as $F$ is generic we have that
$$\Phi_F \colon \PP_F \to {\Theta}_{r,L}$$ is a birational morphism onto its image and that the composition $\theta \circ \phi_F$ is a morphism by proposition 
\ref{thetadivisor}.
We recall that $\theta$ is defined by the determinat line bundle $\mathcal{L}\in \Pic^0(\SU(r,L))$. For simplicity,  we set $\PP^N = |r\Theta_M|$.
\vspace{2mm}

\noindent In order to prove that, for $F$ general, $\theta \circ \Phi_F$ is a linear embedding,  first of all we will prove that $(\theta \circ \Phi_F)^* (\OO_{\PP^N}(1)) \simeq \OO_{\PP_F}(1)$.
\vspace{2mm}

\noindent For any  $\xi \in \Pic^0(C)$  the locus
$$ D_{\xi} = \overline{\{[E] \in \SU(r,L)^s \colon h^0(E \otimes \xi) \geq 1 \}}$$
is an effective divisor in $\SU(r,L)$  and  $\OO_{\SU(r,L)}(D_\xi) \simeq {\mathcal L}$, see \cite{DN}.   

\vspace{2mm}

\noindent Note that 
\begin{equation}
\label{EQ:Dxi}
(\theta \circ \Phi_F)^* (\OO_{\PP^N}(1)) = \Phi_F^*(\theta^*(\OO_{\PP^N}(1))) = \Phi_F^*({\mathcal L}\vert_{\Theta_{r,L}}) = \Phi_F^*(\OO_{\Theta_{r,L}}(D_{\xi})).
\end{equation}
Moreover, one can verify that for general $E \in \Theta_{r,L}^s$ there exists  an irreducible  reduced  divisor $D_{\xi}$  passing  through $E$  such that  
$E $ is a smooth point of the intersection $D_{\xi} \cap {\Theta_{r,L}}$.  
This implies that for general $F$  the pull back  ${\Phi_F}^*(D_{\xi}) $ is a reduced divisor. 
\vspace{2mm}

\noindent Observe that if $\xi$ is such that if $h^1(F \otimes \xi) \geq 1$ (this happens, for example, if $\xi = 0$), then   any extension $E_v$ of $F$ has sections:
$$ h^0 (E_v \otimes \xi) = h^1(E_v \otimes \xi) \geq 1.$$
In particular this implies that $\Phi_F(\PP_F) \subset D_{\xi}$. On the other hand this does not happen for $\xi$ general and we are also able to be more precise about this. Indeed, let $\xi \in \Pic^0(C)$,  then there exists an effective divisor $d \in C^{(2)}$ such that $\xi = {\omega}_C(-d)$. We have that $h^1(F \otimes \xi) \geq 1$ if and only $d \in Z$, where $Z$ is defined  in Lemma \ref{LEM:EVsecant}. Moreover, we can assume that $Z$ is finite by Proposition \ref{LEM:EVsecant} as $F$ is generic.
From now on we will assume that $d \not\in \vert {\omega}_C \vert$ and $d \not\in Z$. 
We can consider the locus
$$ H_{\xi} = \{ [v] \in \PP_F |\quad  h^0(E_v \otimes \xi) \geq 1 \}.$$
We will prove that $H_{\xi}$ is an hyperplane in $\PP_F$ and $\Phi_F^*(D_{\xi}) = H_{\xi}$.
\vspace{2mm}

\noindent From the exact sequence
$$ 0 \to {\xi} \to E_v \otimes \xi \to F \otimes \xi \to 0,$$
passing to cohomology, since $h^0(\xi) = 0$ we have
$$ 0 \to H^0(E_v \otimes \xi)  \to H^0(F \otimes \xi) \to  \cdots$$
from which we deduce that 
$[v] \in H_{\xi}$ if and only if there exists a non zero global section  of 
$H^0(F \otimes \xi)$ which is in the image  of  $H^0(E_v \otimes \xi)$.
Since $d \not\in Z$, then  $h^0(F \otimes \xi) = 1$, let's denote by $s$ a generator of $H^0(F \otimes \xi)$.
\vspace{2mm}

\noindent {\bf Claim}: if $\xi$ is general, we can assume that the zero locus $Z(s)$ of $s$ is actually empty. This can be seen as follows.
By stability of $F \otimes \xi$ we have that $Z(s)$ has degree at most $1$. Suppose that $Z(s)= x$, with $x\in C$. Then we would have an injective map $\OO_C(x) \hookrightarrow F\otimes \xi$  of vector bundles which gives us ${\xi}^{-1}(x)\in \mathcal{M}_1(F)$. Since $F$ is general, if $r \geq 4$ then $\mathcal{M}_1(F)$ is empty by Proposition \ref{maximal} so the zero locus of $s$ is indeed empty. If $r=3$, then 
$$\mathcal{M}_1(F)=\{T_1,\dots,T_m\}$$
is finite. For each $i\in \{1,\dots, m\}$ consider the locus 
$$T_{F,i}=\{\xi\in \Pic^0(C)\, |\, \exists x\in C\, :\, \xi^{-1}(x)=T_i\}.$$ 
This is a closed subset of $\Pic^0(C)$ of dimension $1$. Indeed, $T_{F,i}$ is the image, under the embedding $\mu_i:C\rightarrow \Pic^{0}(C)$ which send $x$ to $T_i(-x)$. Hence the claim follows by choosing $\xi$ outside the divisor $\bigcup_{i=1}^m T_{F,i}$.  
\vspace{2mm}

\noindent As consequence of the claim, we have that $s$ induces an exact sequence of vector bundles 
$$
\xymatrix{
0 \ar[r]& \OO_C\ar[r]^-{\iota_s} & F\otimes \xi \ar[r] & Q\ar[r] & 0.
}
$$
Observe that $[v]\in H_{\xi}$ if and only if $\iota_s$ can be lifted to a map $\tilde{\iota_s}:\OO_C\rightarrow E\otimes \xi$. Then,
by Lemma \ref{lift}, we have  that $H_{\xi}$ is actually the projectivization of the kernel of the  following map: 
$$H^1(\iota_s^*): H^1(\mHom(F\otimes\xi,\xi))\rightarrow 
H^1(\mHom(\OO_C,\xi))$$
which proves that $H_{\xi}$ is an hyperplane as $H^1(\iota_s^*)$ is surjective and 
$$H^1(\mHom(\OO_C,\xi)) \simeq H^1(\xi)\simeq \C.$$
Note  that we have the inclusion  ${\Phi_F}^*(D_{\xi}) \subseteq H_{\xi}$.  Since 
both are effective divisors and $H_{\xi}$ is irreducible we can conlude that they  have the same support. Finally,   since  ${\Phi_F}^*(D_{\xi})$ is reduced, then they are the same divisor. 
In particular, as claimed, we have 
$$\Phi_F^*(\OO_{\Theta_{r,L}}(D_{\xi}))=\OO_{\PP(F)}(1).$$
In order to conclude we simply need to observe that the map is induced by the full linear system $|\OO_{\PP_F}(1)|$. But this easily follows from the fact that $\theta\circ \Phi_F$ is a morphism. Hence $\theta\circ \Phi_F$ is a linear embedding and the Theorem is proved.
\end{proof}

\begin{remark}
The above Theorem implies that ${\Phi}_L^*({\mathcal L})$ is a unisecant line bundle on the projective bundle $\PP(\mV_L)$. 
\end{remark}

\end{document}